\documentclass[12pt]{amsart}
\usepackage{amsfonts, amsmath, amstext, amssymb, xypic, enumerate, hyperref, fullpage}
\usepackage{mathrsfs}

\usepackage[dvips]{graphicx}
\usepackage{latexsym}
\usepackage[dvipsnames]{xcolor}

\newtheorem{thm}{Theorem}[section]
\newtheorem{prop}[thm]{Proposition}

\newtheorem{lemma}[thm]{Lemma}

\newtheorem{cor}[thm]{Corollary}

\newtheorem{definitiontemp}[thm]{Definition}
\newenvironment{defn}{\begin{definitiontemp}
\normalfont}{\end{definitiontemp}}

\def\bec{\begin{cor}}
\def\enc{\end{cor}}

\def\bet{\begin{thm}}
\def\ent{\end{thm}}

\def\becor{\begin{cor}}
\def\encor{\end{cor}}

\def\bel{\begin{lem}}
\def\enl{\end{lem}}

\def\bedef{\begin{defn}}
\def\endef{\end{defn}}

\def\bep{\begin{prop}}
\def\enp{\end{prop}}

\newcommand{\set}[2]{\ensuremath{ \{ #1 : #2 \} }}

\newcommand{\dg}{\ensuremath{\textnormal{deg}}}

\newcommand{\mdeg}[1]{\ensuremath{\textnormal{mdeg}(#1)}}
\renewcommand{\dim}[1]{\ensuremath{\textnormal{dim}(#1)}}
\newcommand{\rk}[1]{\ensuremath{\textnormal{rk}(#1)}}

\newcommand{\Z}{\mathbb{Z}}

\newcommand{\Q}{\mathbb{Q}}
\newcommand{\ZZ}{\mathbb{Z}}
\newcommand{\QQ}{\mathbb{Q}}

\newcommand{\OO}{\mathcal{O}}

\newcommand{\U}{\mathcal{U}}

\renewcommand{\hbar}{\overline{h}}

\newcommand{\Qbar}{\overline{\mathbb{Q}}}

\renewcommand{\L}{\mathcal{L}}

\newcommand{\avec}{\vec{a}}
\newcommand{\bvec}{\vec{b}}
\newcommand{\fvec}{\vec{f}}

\newcommand{\yvec}{\vec{y}}
\newcommand{\Yvec}{\vec{Y}}

\newcommand{\gq}{{\mathfrak q}}
\newcommand{\pp}{{\mathfrak p}}

\newcommand{\comment}[1]{}

\newcommand{\Gal}[2]{\text{Gal}(#1/#2)}

\DeclareMathOperator{\Frac}{Frac}

\DeclareMathOperator{\Spec}{Spec}
\DeclareMathOperator{\trdeg}{trdeg}
\DeclareMathOperator{\tr}{tr}

\def\s01{\ensuremath{\Sigma^0_1}}
\def\d02{\ensuremath{\Delta^0_2}}
\def\phi{\varphi}

\newcommand{\Subfields}{\operatorname{Sub}(\Qbar)}

\begin{document}

\title{A topological approach to undefinability\\
	 in algebraic extensions of $\QQ$}

\author{Kirsten Eisentr\"ager}
\address{
\parbox{\dimexpr\textwidth -\parindent}{ \hspace*{0.1in}Department of Mathematics, The Pennsylvania State University, University Park, PA 16802, USA}}
\email{kxe8@psu.edu}

\author{Russell Miller}
\address{
\parbox{\dimexpr\textwidth -\parindent}{ \hspace*{0.1in}  Department of Mathematics, Queens College -- City University of New York, 65-30 Kissena Blvd., Flushing, NY 11367 USA \\
\hspace*{0.09in}  Ph.D. Programs in Mathematics and Computer Science, Graduate Center - City University of New York, 365 Fifth Avenue, New York, NY 10016 USA}}
\email{Russell.Miller@qc.cuny.edu}

\author{Caleb Springer}
\address{
\parbox{\dimexpr\textwidth -\parindent}{ \hspace*{0.1in}Department of Mathematics, University College London, London WC1H 0AY, UK\\ 
\hspace*{0.1in}The Heilbronn Institute for Mathematical Research, Bristol, UK}}
\email{c.springer@ucl.ac.uk}

\author{Linda Westrick}
\address{
\parbox{\dimexpr\textwidth -\parindent}{ \hspace*{0.1in}Department of Mathematics, The Pennsylvania State University, University Park, PA 16802, USA}}
\email{westrick@psu.edu}

\begin{abstract}

For any subset $Z \subseteq \Q$, consider the set $S_Z$ of subfields $L\subseteq \Qbar$ which contain a co-infinite subset $C \subseteq L$ that is universally definable in~$L$ such that $C \cap \Q=Z$.
Placing a natural topology on the set $\Subfields$ of subfields of $\Qbar$, we show that if $Z$ is not thin in~$\Q$, then $S_Z$ is meager in $\Subfields$.
Here, \emph{thin} and \emph{meager} both mean ``small", in terms of arithmetic geometry and topology, respectively.
For example, this implies that only a meager set of fields $L$ have the property that the ring of
algebraic integers $\OO_L$ is universally definable in~$L$.
The main tools are Hilbert's Irreducibility Theorem and a new
normal form theorem for existential definitions.  The normal form theorem,
which may be of independent interest, says roughly that every $\exists$-definable subset
of an algebraic extension of $\mathbb Q$ is a finite union of single points and projections
of hypersurfaces defined by absolutely irreducible polynomials.
\end{abstract}
\thanks{
algebraic fields, algebraic integers, definability, Hilbert Irreducibility Theorem, Hilbert's Tenth Problem}

\maketitle

\section{Introduction}
\label{sec:intro}

Let $\Subfields$ denote the set of subfields of $\Qbar$.  Given a field $L \in \Subfields$
and a set $C \subseteq L$, it is a question of general interest whether $C$ is 
first-order definable 
in $L$ using the language of rings.  If so, one also wants to know 
how simple a defining formula can be.  For example, 
results of Koenigsmann~\cite{K16}, extended by Park~\cite{P13}, have shown that in every number field $K$,
the ring $\mathcal O_K$ of algebraic integers is defined by a universal formula.
Here we show that the usual situation is the opposite,
not only for rings of integers but for any subset $A \subseteq \Qbar$
satisfying a rather general condition on $A \cap \mathbb Q$.
Just as $\OO_L = \OO_{\Qbar}\cap L$, we write
$A_L = A\cap L$.
Placing a natural topology on $\Subfields$, we will show that in most cases
there is a comeager set of fields $L \in \Subfields$ such that
$A_L$ cannot be defined in $L$ by any universal formula.

\begin{thm}
\label{thm:intro-fullygeneral}
If $A\subseteq \Qbar$ is a subset for which $A_{\Q}$ is coinfinite and not thin (as a subset of the Hilbertian field $\Q$), then the following class is meager in $\Subfields$:
$$ \U_A=\set{L\in\Subfields}{A_L\text{~is universally definable in~}L}.$$
\end{thm}
 Indeed, a stronger statement holds, and depends only on the subset of $\Q$ in question.


\begin{thm}[Theorem \ref{thm:fullygeneral2}]
\label{thm:intro-fullygeneral}
If $Z\subset\Q$ is not thin, then the following is meager in $\Subfields$:
$$	S_Z  = \bigcup_{A\subseteq \Qbar \colon A\cap \Q=Z} \{L \in \Subfields : A_L \text{ is coinfinite and universally definable in }L\} $$
\end{thm}

The second theorem implies the first by setting $Z=A_{\Q}$.
Thus the irrational portion of $A$ is irrelevant:  in all subfields $L$
outside the meager class $S_{A_\Q}$, neither $A_L$ nor any other set
that intersects $\Q$ in $A_\Q$ can be universally defined.
Clearly this is much stronger than the first statement.

Dually (with $B=\Qbar\setminus A$), if $B_{\Q}$ is infinite and not co-thin in $\Q$,
then the class
$$\mathcal E_B= \set{L\in\Subfields}{B_L\text{~is existentially definable in~}L}$$
equals $\U_A$, hence is meager.  The second statement can also be applied
in a dual form to existentially definable sets.


The notion of thinness which appears in the theorem is due to Serre.
Intuitively, a set is thin if it is ``small" in the sense of arithmetic geometry; see Section \ref{subsec:thin}.
Initially we did not expect definability of a set to be intertwined
with any notion of its size apart from finiteness, but this condition arose naturally
in our investigations.

The topology on $\Subfields$ is defined by considering it as a subset of the power set $2^{\Qbar}$, 
from which it inherits the product topology.
In this topology, every nonempty open set is non-meager.  
The topology also coincides (via the Galois 
 correspondence) with the Vietoris topology on the space of closed 
 subgroups of $\operatorname{Gal}(\Qbar/\Q)$.
 We thank Florian Pop for pointing out this connection to us, and for
 alerting us that the same topology appears
 in \cite{Pop1999}, where it is called the strict topology.
 The topology has also been used by other authors: for examples, see
 \cite{HaranJarden1988, Ershov1995, GartsideFisher2009, GartsideSmith2010}.

The theorem also remains true
when replacing $\Subfields$ with the quotient space
$\Subfields/\!\cong$ considered in \cite{M19}, which only considers fields up to isomorphism; see Corollary \ref{cor:up_to_isom}.

Using the fact that neither $\mathbb Z$ nor $\mathbb Q \setminus \mathbb Z$ is 
thin in $\mathbb Q$, we obtain the following corollary.

\begin{thm}[Theorem \ref{thm:meager_2}]
\label{thm:intro-main}
The set of algebraic extensions $K$ of $\QQ$ for which $\OO_K$ is existentially or universally definable is a meager subset of $\Subfields$.
 \end{thm}

After seeing one of the authors speak on these results, Philip Dittmann
and Arno Fehm extended Theorem \ref{thm:intro-main} in a different way,
in \cite{DittmannFehm2021},
improving ``existentially or universally definable'' to ``definable'' by explicitly  
using the fact that $\OO_K$ forms a ring.  Their proof
uses techniques from model theory, entirely different from those 
employed here.

 \subsection{Outline of the paper} To prove Theorem \ref{thm:intro-fullygeneral}, 
 we study the existential definability of
 sets $Y\subseteq \QQ$ whose complement is not thin, in the
 sense of Serre.  These are the complements of the sets $Z$ described
 above.  The necessary background of algebraic number theory,
 arithmetic geometry and thin sets is recalled in Section
 \ref{sec:ANT}.  In order to prove the main theorem, we introduce a
 new notion of rank in Section \ref{sec:rank} that applies to
 existential formulas.  This notion generalizes the multidegree of
 a polynomial in a way that we found to be both natural and quite useful,
 providing a
 pre-well-ordering of existential formulas.  Thus, if $Y$ is existentially
 definable within $\QQ$ over some field $L\subseteq \Qbar$, then there is
 a formula of least rank which does the job.  By studying such
 minimal-rank formulas in Section \ref{sec:hyp}, we obtain the following
 normal form for existential definitions, which may be of independent 
 interest.
  
  \begin{thm}[(Theorem \ref{thm:normal-form})]
 For any field $L \subseteq \Qbar$, if $A \subseteq L$ is existentially definable in $L$,
then $A$ is definable in $L$ by a formula of the form 
$$\alpha(X) = \vee_{i =1}^r \beta_i(X),$$ where each $\beta_i(X)$ has one of the following forms:
 \begin{enumerate}[(i)]
	 \item The quantifier-free formula $X=z_0$ for a fixed $z_0\in L$.
	 \item A formula of the form 
 $$\exists Y_1\dots \exists Y_e~[ f(X, Y_1, \dots, Y_e) = 0 \neq g(X, Y_1, \dots, Y_e)]$$ 
 for polynomials $f,g \in L[X, Y_1, \dots, Y_e]$,
	 where $f$ is absolutely irreducible and does not divide $g$.
\end{enumerate}
  \end{thm}
 
 Finally, we introduce the topological spaces
 of $\Subfields$ and $\Subfields/\!\cong$ in Section
 \ref{sec:main_thm}, and use the normal form to deduce the main result
 via Hilbert's Irreducibility Theorem.  In fact, the proof also leads
 to an algorithm which, given a basic open subset
 $U\subseteq \Subfields$, produces a computable field $L\in U$ in
 which the ring of integers $\OO_L$ is neither existentially or
 universally definable; see Theorem \ref{thm:computable_fields}.

 \subsection{Previous work on definability of rings of integers}
 Much of the previous work on the definability of subsets $A \subseteq K \in \Subfields$
 has focused on the case where $A = \OO_K$.  We conclude the 
 introduction with a overview of the literature on this case.
 
The existential definability of $\OO_K$ in $K$
 is an ingredient that would assist a standard reduction argument
for proving undecidability results for generalizations of Hilbert's
Tenth Problem.  In its original form, this problem asked for an algorithm that
decides, given a polynomial equation $f(x_1,\dots,x_n)=~0$ with
coefficients in the ring $\Z$ of integers, whether there is a solution
with $x_1,\dots,x_n \in \Z$.  Matiyasevich \cite{Mat70}, building on
earlier work by Davis, Putnam, and Robinson \cite{DPR61}, proved that
no such algorithm exists, i.e., Hilbert's Tenth Problem is
undecidable.  Since then, analogues of this problem have been studied
by asking the same question for polynomial equations with coefficients
and solutions in other recursive commutative rings.  One of the most
important unsolved questions in this area is Hilbert's Tenth Problem
over the field of rational numbers $\Q$, and more generally over
number fields.  If $\mathbb{Z}$ is existentially definable in
$\mathbb{Q}$, then a reduction argument shows that Hilbert's Tenth
Problem for $\mathbb{Q}$ must be undecidable.

However, if Mazur's Conjecture holds, then $\mathbb{Z}$ is 
not existentially definable in $\mathbb{Q}$.
Proving this unconditionally currently appears to
be out of reach.  In fact, it seems generally very difficult
to prove undefinability results for individual fields. One example of
success is the field of all totally real algebraic numbers
$\QQ^{\tr}$. Fried, Haran and V\"olklein showed that its first-order
theory is decidable \cite{FHV94}, while J.\ Robinson showed that the
first-order theory of the ring of all totally real integers
$\ZZ^{\tr}$ is undecidable \cite{R62}.  This difference in
decidability implies that $\ZZ^{\tr}$ cannot be first-order definable
in the field $\QQ^{\tr}$.  Another example is the ring $\overline\ZZ$
of all algebraic integers inside $\Qbar$, which is undefinable by the
strong minimality of $\Qbar$.  In both examples, the facts used for
proving undefinability are not remotely close to necessary conditions
for undefinability. Instead, they simply reflect the available
pathways for unconditionally proving undefinability in a limited
number of cases.

While it is still an open question whether $\mathbb{Z}$ is
existentially definable in $\mathbb{Q}$, it is possible to give a
first-order definition of $\ZZ$ in $\QQ$, i.e.\ a definition that uses both existential
and universal quantifiers. 
This was first done by J.\ Robinson~\cite{Rob49}, who generalized this result to define
the ring of integers $\mathcal{O}_K$ inside any number field $K$~\cite{Rob59}. Later,
Rumely~\cite{Rum80} was able to make the definition of the ring of
integers uniform across  number fields. Robinson's definition was
improved by Poonen \cite{Poon09} who gave a $\forall \exists$-definition that in every
number field $K$ defines its ring of integers.  
Following this, Koenigsmann \cite{K16} proved that it is possible to give a universal
definition of $\mathbb{Z}$ in $\mathbb{Q}$, i.e.\ a definition that
only involves universal $(\forall)$ quantifiers, and Park extended his result
to show that $\OO_K$ is universally definable in $K$ for every number field
$K$~\cite{P13}.  This raises the question of whether we can expect
universal and first-order definability to continue to hold for many infinite algebraic extensions of
$\QQ$.

Currently, first-order definability results are only known for certain
classes of infinite extensions of the rationals. These are usually
proved in order to establish the first-order undecidability of certain
infinite extensions via reductions.  For example, Videla proved the
definability of the ring of integers over certain infinite algebraic
\hbox{pro-$p$} extensions of $\mathbb{Q}$ \cite{Vid00}, while Fukuzaki
was able to define the ring of integers in infinite extensions in
which every finite subextension has odd degree and that satisfy
certain ramification conditions~\cite{Fuk2012}.  These results were
further generalized by Shlapentokh in~\cite{Shl2018}, to which we
refer readers for more extensive background on known results for the
first-order definability and decidability of infinite algebraic
extensions of $\QQ$.  In Shlapentokh's framework, all known examples
of algebraic extensions of $\QQ$ with first-order definable rings of
integers can be viewed as relatively small extensions which are
somehow ``close" to $\QQ$.  On the other hand, although first-order
definability seems less likely for extensions which are similarly
``far from" $\QQ$, very few negative examples are known, as mentioned
above.

\subsection*{Acknowledgements}
This project began during a workshop at the American Institute of
Mathematics in May 2019.  It is based upon work supported by the
National Science Foundation under Grant \# DMS-1928930 while the
authors participated in a program hosted by the Mathematical Sciences
Research Institute in Berkeley, California, during the Fall 2020
semester. The authors wish to acknowledge useful conversations with
Tom Tucker.  Eisentr\"ager was partially supported by National Science
Foundation awards CNS-1617802, CNS-2001470, and a Vannevar Bush Faculty Fellowship
from the US Department of Defense.  Miller was partially supported by
Grant \# 581896 from the Simons Foundation and by the City University
of New York PSC-CUNY Research Award Program.  Springer was partially
supported by National Science Foundation award CNS-1617802.  Westrick
was partially supported by the Cada R. and Susan Wynn Grove Early
Career Professorship in Mathematics.

\section{Background from number theory and algebraic geometry}
\label{sec:ANT}

In this section, we will recall some of the basic facts that we will
require for fields, thin sets, and affine varieties.  Readers can find
additional background in the books of Fried and Jarden \cite{FJ08}, Lang \cite{lang}, Serre
\cite{tigt} and Liu \cite{Liu}.

\subsection{Field extensions and the irreducibility of polynomials}

In the material that follows, we will be presented with the following question: Given number fields $F \subseteq K$, which field extensions of $F$ contain elements of the complement $K\setminus F$? This question is intimately related to the  irreducibility of polynomials.  First, we recall a basic result on the irreducibility of multivariable polynomials.


\begin{lemma}
\label{lemma:irreducibles}
If $K/F$ is an extension of fields within a larger field $L$, and $z\in L$ is algebraic over $F$
with $F(z)\cap K\neq F$, then the minimal polynomial $h(Z)$ of $f$ over $F$ must be reducible over $K$.
\end{lemma}
\begin{proof}
By hypothesis $1<[F(z)\cap K:F]$, so
$$ [F(z):F(z)\cap K] < [F(z):F(z)\cap K]\cdot [F(z)\cap K:F] = [F(z):F].$$
From this it follows that $h(Z)$ must factor over $F(z)\cap K$, so it certainly also
factors over the larger field $K$.
\end{proof}

The next proposition forms a kind of converse to Lemma \ref{lemma:irreducibles}
when $K/F$ is a finite Galois extension.  Given an algebraic function field $E = \Frac(F[Y_0,Y_1,\ldots,Y_m]/(f))$ where $f\in F[Y_0,Y_1,\ldots,Y_m]$ is an irreducible polynomial, the \emph{constant field} of $E$ is the set of elements which are algebraic over $F$.

\begin{prop}
\label{prop:missinglemma}
Let $F$ be a number field, and $K$ a finite Galois extension of $F$.
If $m\geq 0$ and $f\in F[Y_0,Y_1,\ldots,Y_m]$ is an irreducible polynomial
that becomes reducible in $K[Y_0,Y_1,\ldots,Y_m]$, then the
constant field of $E = \Frac(F[Y_0,Y_1,\ldots,Y_m]/(f))$ is larger than $F$.
 In particular, there is an element $z\in E\setminus F$ such that  there is an $F$-linear field embedding of $F(z)$ into $K$ with the image of $z$ lying in $K\setminus F$.
\end{prop}

 \begin{proof} 
 Assume without loss of generality that $Y_m$ appears nontrivially in $f$, and write $L = F(Y_0, Y_1,\dots, Y_{m-1})$.   
 We will view $E = L(\theta)$ for an element $\theta$ in the algebraic closure $\overline L$ with minimal polynomial $f$.  Similarly consider $K$ to be an extension of $F$ inside $\overline L$.

Suppose that $E$ contains no elements of $K\setminus F$.  Then $E\cap K = L \cap K = F$, and a basic theorem of Galois theory \cite[Theorem 1.12]{lang} implies the following because $K$ is a Galois extension of $F$:
$$
	[EK : E] = [K : E\cap K] = [K : F] = [K : L\cap K] = [LK : L].
$$

Using the diamond written below, we deduce that $[E : L]  = [EK : LK]$.  
Importantly, these field extension degrees are also the degrees of the minimal polynomial of $\theta$ over $L$ and $LK$, respectively.  
   
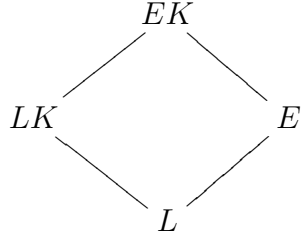
\begin{align*} 
   \xymatrix{& EK  \ar@{-}[dr] \ar@{-}[dl] & \\
           LK \ar@{-}[dr]  & & E \ar@{-}[dl]\\
           & L&}
          \end{align*}

This shows that $f$ remains irreducible over the field ${L = K(Y_0,Y_1,\dots, Y_{m-1})}$ as a polynomial in $Y_m$.   
We claim that $f$ is actually irreducible as an element of the ring $K[Y_0, Y_1,\dots, Y_{m}]$, which contradicts the hypothesis. To prove this, it only remains to show that the coefficients of $f$ lying in $K[Y_0,\dots, Y_{m-1}]$ have no common factor; see \cite[IV.2.3]{lang}. 
 Clearly, as a polynomial in $Y_m$, the coefficients of $f$ lying in $F[Y_0, \dots, Y_m]$ have no common factor over $F$ because $f$ is irreducible over $F$. 
  In fact, this implies that the coefficients also have no common factor over any algebraic extension of $F$ by the following lemma, which completes the proof.
 \end{proof}
 
 \begin{lemma} 
 Let $F$ be a field and let $F'$ be a separable extension. If $f_0, f_1, \dots, f_k$  are a collection of polynomials in $F[Y_0, \dots, Y_m]$ with no common factor, then $f_0,\dots, f_k$ also have no common factor over the extension $F'$.
  \end{lemma}
   \begin{proof} 
By writing $f_0, \dots, f_k$ in terms of their irreducible factors, we can reduce without loss of generality to the case of two irreducible polynomials $f_0, f_1\in F[Y_0, \dots, Y_m]$.  Indeed, for every irreducible factor $p$ of $f_0$, there is a polynomial $f_j$ for $1\leq j\leq m$ which is not divisible by $p$, and it suffices to show that the irreducible factors of $f_j$ remain relatively prime to $p$ over the larger field $F'$.

Notice that irreducible polynomials $f_0$ and $f_1$ are relatively
prime over $F$ if and only if $f_0f_1$ generates a radical ideal in
$F[Y_0,\dots, Y_m]$, i.e.\ if and only if
$F[Y_0, \dots, Y_m]/(f_0f_1)$ is a reduced ring.  The latter condition
is stable under separable field extensions, i.e.\
$F'[Y_0, \dots, Y_m]/(f_0f_1)$ is also reduced; see \cite[Proposition
3.2.7.(b)]{Liu}. Therefore $f_0$ and $f_1$ have no common factors over
$F'$.
   \end{proof}

\subsection{Dimensions of rings and affine varieties}

\label{sec:dim}

We will require a usable notion of dimension, which can equivalently
be viewed as a geometric or algebraic phenomenon.  In particular,
there are related notions of the dimension of a commutative ring $A$,
and the dimension of the associated topological space $\Spec A$
consisting of all prime ideals of $A$ with the Zariski topology.  In
this section, we will review some basic facts of commutative algebra
and algebraic geometry, limiting the discussion to only what is
necessary for our purposes.

First, let us recall this topology and some basic notation.  Given a
commutative ring $A$, the set $\Spec A$ is endowed with the
\emph{Zariski topology} by defining the following as basic closed and
open sets, respectively.  For any ideal $I\subseteq A$, we define
$V(I)$ to be the subset of $\Spec A$ consisting of all prime ideals
that contain $I$, and $D(f) = \Spec A \setminus V(f)$.  Notice that
it is natural via the isomorphism theorems for rings to identify
$V(I)$ with $\Spec A/I$.  With this notation, the closed subsets of
$\Spec A$ in the Zariski topology are precisely the sets of the form
$V(I)$ where $I\subseteq A$ is an ideal, and sets of the form $D(f)$
for $f\in A$ form a base for the open subsets of $\Spec A$. In fact,
$\Spec A$ is an \emph{affine scheme}, meaning that it has even more
structure than just a topology, although we will not require this full
structure; see \cite[Chapter 2]{Liu} for more background.

In this paper, we consider the ring $A = F[Y_0, \dots, Y_m]$ and its
quotients, where $F$ is a subfield of $\overline \QQ$.  An
\emph{affine variety over F} is an object of the form
$V(I) = \Spec F[Y_0,\dots, Y_m]/I$ for some $m\geq 0$ and some ideal
$I\subseteq F[Y_0,Y_1,\dots, Y_m]$.  Furthermore, if the quotient
$F[Y_0,\dots, Y_m]/I$ is an integral domain, then the corresponding
affine variety is called \emph{integral}.  We will write
$V(I) = V(f_1, \dots, f_k)$ when the ideal
$I\subseteq F[X,Y_1,\dots, Y_m]$ is generated by $\{f_1,\dots, f_k\}$.
If there is ambiguity about the base field, then we will write $V_F$
instead of $V$ for clarity.

Given an affine variety $V = \Spec F[Y_0,\dots, Y_m]/I$, 
the \emph{rational points} of $V$ (over $F$) are the tuples
$(y_0,\dots, y_m)\in F^m$ such that $f(y_0, \dots, y_m) = 0$ for all
$f\in I$. The set of rational points can be identified with the set of
all $F$-algebra homomorphisms $\varphi: F[Y_0,\dots, Y_m]/I\to F$.  We
refer the reader to \cite[Section 2.3.2]{Liu} for more details.
As we are frequently working over non-algebraically closed
fields, it is possible for nontrivial affine varieties to have no
rational points, such as the affine variety
$\Spec\QQ[Y_0, \dots, Y_m]/(Y_0^2 + \dots + Y_m^2 +1)$ for any
$m \geq 0$.  We can view the varieties as geometric objects which help
us find and describe the rational points.

The {\emph{Krull dimension}} of a ring $A$, written $\dim A$, is the supremal
length $r$ of a chain of prime ideals
$\pp_0 \subsetneq \dots \subsetneq \pp_r$ in $A$.  Similarly, given a
topological space $X$, we define $\dim X$ to be the supremal length
$r$ of a chain of irreducible closed subsets
$Z_0\subsetneq \dots \subsetneq Z_r$ in $X$.  The following proposition
equates these two notions of dimension. Recall that the nilradical of
a commutative ring is the set of all nilpotent elements, or
equivalently the intersection of all prime ideals.

\begin{prop}[Proposition 2.5.8, \cite{Liu}]
Let $A$ be a (commutative) ring and let $N$ be the nilradical of $A$. Then $\dim{\Spec A} = \dim{A} = \dim{A/N}$.
 \end{prop}
 
 In our applications, we need to understand the dimension of subsets of affine varieties. Recall that if $X$ is any topological space and $Y$ is any subset of $X$ endowed with the subset topology, then $\dim Y\leq \dim X$ \cite[Proposition 2.5.5]{Liu}.  In the context of affine varieties and open subsets, this inequality is often an equality due to the fact that open subsets in the Zariski topology are ``large". This idea is formulated precisely in the following proposition.  Given a field extension $L/F$, we write $\trdeg_F L$ for the \emph{transcendence degree} of $L$ over $F$.   If $X = \Spec A$ is an integral affine variety, we call $\Frac(A)$ the \emph{function field} of $X$.
 
 \begin{prop}[Proposition 2.5.19, \cite{Liu}]
 \label{prop:open}
 If $X = \Spec A$ is an integral affine  variety over a field $F$, then 
 $$
 	\dim U 
		= \dim X 
		= \trdeg_F \Frac(A)
$$
 for each nonempty open subset  $U\subseteq X$.
  \end{prop}
  
  Similarly, it is helpful to know when a subset of a topological space $X$ has strictly smaller dimension than $X$. In contrast to the result immediately above, this often happens for proper closed subsets of an affine variety. 
  
  \begin{prop}[Corollary 2.5.26, \cite{Liu}]
  \label{prop:closed}
  Let $X = \Spec A$ be an integral affine variety.  If $f\in A$ is nonzero, then every irreducible component of $V(f)$ has dimension $\dim X - 1$. In particular, every proper closed subset of $X$ has strictly smaller dimension than $X$.
   \end{prop}
  
  So far in this section, the definition of dimension depends on the base field $F\subseteq \overline\QQ$, \emph{a priori}. However, the result below clarifies that dimension stays the same under base extension.  This allows us to ignore the field of definition to some extent, especially when defining the rank of a formula below, although the notion of integrality truly does depend on the base field, so care is still required when applying the previous two propositions.
  
  \begin{prop}[Proposition 3.2.7, \cite{Liu}]
  \label{prop:base-ext}
Let $F \subseteq L\subseteq \Qbar$ be fields.  Given an affine variety $V_F(f_1,\dots, f_k) = \Spec F[Y_0,\dots, Y_m]/(f_1,\dots, f_k)$, the affine variety 
$$
	V_L(f_1,\dots, f_k) =  \Spec L[Y_0,\dots, Y_m]/(f_1,\dots, f_k)
$$
 is the \emph{base extension} of the variety $V_F(f_1,\dots, f_k)$ to $L$, and these affine varieties have the same dimension.
   \end{prop}
   
   To apply this proposition to open sets, we remark that open sets can be equivalently viewed as affine varieties themselves, albeit in a different ambient space with an extra variable.

   \begin{cor} \label{cor:localization} Let $F\subseteq \Qbar$ be a
     field.  For polynomials
     $g, f_1, \dots, f_k\in F[Y_0, \dots, Y_m]$, define
     $A = F[Y_0, \dots, Y_m]/(f_1, \dots, f_k)$ and let $A_g$ be the
     localization of $A$ be the element $g$.  Then there are
     isomorphisms of ringed topological spaces
$$
	V_F(f_1, \dots, f_k)\cap D(g) \cong \Spec(A_g) \cong V_F(f_1, \dots, f_k, Y_{m+1}g - 1).
$$
In particular, $\dim{V_F(f_1, \dots, f_k)\cap D(g)} = \dim{V_L(f_1, \dots, f_k)\cap D(g)}$ for any algebraic extension of fields $L\supseteq K$.
 \end{cor}
  \begin{proof} 
  The first isomorphism is \cite[Lemma 2.3.7]{Liu}.  The second isomorphism actually follows from a well-known isomorphism of underlying rings 
  $$
  	A_g \cong F[Y_0,\dots, Y_{m+1}]/(f_1, \dots, f_k, Y_{m+1}g - 1);
$$
see \cite[Lemma \S6.2]{R95}.  Therefore, the statement on dimension follows immediately from Proposition \ref{prop:base-ext}.
\end{proof}

\subsection{Thin sets}
\label{subsec:thin}
Hilbert's Irreducibility Theorem can take many different forms, but we
put a simple version here that suffices for the purposes of this
article.  For brevity, we present \emph{thin sets} as a black box, and
refer the reader to \cite[Prop.\ 3.3.5]{tigt} for more details.
Essentially, a thin subset $T\subseteq K$ of a number field is small,
in the view of arithmetic geometry. For example, any set of points
that is contained in a closed subvariety of affine $n$-space $K^n$,
and which is different from the entire space, is thin with respect to
$K$. All necessary details can be deduced from the results we recall
below .

\begin{thm}[Hilbert's Irreducibility Theorem] \label{HIT} Let
  $f(Y_0,Y_1,\dots, Y_m)$ be a polynomial with coefficients in a
  number field $K$ which is irreducible as an $(m+1)$-variable
  polynomial.  There exists a thin set $T\subseteq K^m$ such that if
  $(y_1,\dots, y_m)\in K^m \setminus T$, then $f(Y_0,y_1,\dots, y_m)$
  is an irreducible single-variable polynomial of degree
  $\dg_{Y_0}(f)$.
 \end{thm}
  
 In order for the theorem above to be non-trivial, we need to know
 that $K^m$ is not a thin subset of itself, and this is indeed true
 for all number fields \cite[Prop 3.4.1]{tigt}.  Moreover, the
 propositions below show that thin sets cannot contain arithmetically
 important subsets, which will allow us to use Hilbert's
 Irreducibility Theorem in the cases we care about.

\begin{prop}[Proposition 3.2.1, \cite{tigt} ]
\label{prop:fin-int}
If $L/K$ is a finite extension of fields and $T\subseteq L^m$ is thin with respect to $L$, then $T\cap K^m$ is thin with respect to $K$.
 \end{prop}

\begin{prop}
\label{thin}
If $K$ is a number field, then no thin subset of $K$ contains either $\ZZ$ or $\QQ\setminus \ZZ$.
 \end{prop}
  \begin{proof} 
    Thin sets of $\QQ$ cannot contain $\ZZ$ or $\QQ\setminus \ZZ$ by
    \cite[Theorem 3.4.4]{tigt} and \cite[Prop. 3.4.2]{tigt},
    respectively.  Thus, the result for arbitrary number fields
    follows from Proposition \ref{prop:fin-int}.
\end{proof}

Moreover, we can understand thin sets in products. This lemma will be used to show that if a set $Z\subseteq \QQ$ is not thin, then the product $Z\times \QQ^n$ cannot be thin, either.

\begin{lemma} 
\label{lem:thin_prod}
If $n\geq 0$ and $S\subseteq \QQ$ is a set such that $ S\times \QQ^n \subseteq \QQ^{n+1}$ is thin, then $S\subseteq \QQ$ is thin.
 \end{lemma}
\begin{proof}
  There is a line $\mathcal L\subseteq \QQ^{n+1}$ such that
  $\mathcal L \cap (S\times \QQ^n)$ is thin in $\mathcal L$
  and the projection of $\mathcal L$ to the first coordinate is
  all of $\QQ$ \cite[Proposition 3.2.3]{tigt}.  As $\mathcal L$ is a
  line, this projection is an isomorphism and
  $\mathcal L \cap (S\times \QQ^n)$ maps onto to the set $S$.
  Therefore, $S$ is thin in $\QQ$.
\end{proof}

Finally, we prove a proposition that lets us stitch this material together.  This is ultimately the result that is required in the proof of our main theorem.

\begin{prop} 
\label{prop:hit2}
Let $K$ be a number field and let $f(X, Y_1, \dots, Y_m), g(X, Y_1, \dots, Y_m) \in K[X, Y_1,\dots, Y_m]$ be relatively prime irreducible polynomials. Then there is a thin set $T\subseteq K^m$ such that $f(x, y_1, \dots, y_{m-1}, Y)$ and $g(x, y_1, \dots, y_{m-1}, Y)$ are relatively prime irreducible single-variable polynomials for every $(x,y_1,\dots, y_{m-1})\in K^m\setminus T$, of degrees $\dg_{Y_m}(f)$ and $\dg_{Y_m}(g)$, respectively.
 \end{prop}
  \begin{proof} 
    Take $T_0$ to be the union of the two thin sets given by applying
    Hilbert's Irreducibility Theorem to $f$ and $g$ separately.  By
    construction, $f(x, y_1, \dots, y_{m-1}, Y)$ and
    $g(x, y_1, \dots, y_{m-1}, Y)$ are irreducible polynomials in $Y$
    for every $(x,y_1,\dots, y_{m-1})\in K^m\setminus T_0$, and it
    only remains to check the claim of relative primality.

If $\dg_{Y_m}(f) \neq \dg_{Y_m}(g)$, then this claim is trivial.  Therefore, write $d = \dg_{Y_m}(f) = \dg_{Y_m}(g)$, and consider $(x,\dots, y_{m-1})\in K^m\setminus T$.  Since the polynomials $f(x, y_1, \dots, y_{m-1}, Y)$ and $g(x, y_1, \dots, y_{m-1}, Y)$ are irreducible, the failure of relative primality implies that they are unit multiples of each other, i.e.,
$f(x, y_1, \dots, y_{m-1}, Y) = zg(x, y_1, \dots, y_{m-1}, Y)$ for some nonzero $z\in K$.  In particular, if we write 
$$
	f(X, Y_1,\dots, Y_m) = \sum_{i = 0}^{d} f_i(X, Y_1, \dots, Y_{m-1})Y_m^i,
$$
$$
	g(X, Y_1,\dots, Y_m) = \sum_{i = 0}^{d} g_i(X, Y_1, \dots, Y_{m-1})Y_m^i,
$$
where $f_i, g_i\in K[X, Y_1,\dots, Y_{m-1}]$ are polynomials,
then this condition is the same as 
$$
	f_i(x, y_1, \dots, y_{m-1}) = zg_i(x, y_1, \dots, y_{m-1})
$$ for all $0\leq i\leq d$.  Multiplying these conditions together, we get the equations
$$
	f_ig_j = zg_ig_j = g_if_j
$$
for $0\leq i, j \leq d$.  We will show that this system of equations holds only inside a thin set, which completes the proof.

We claim that the polynomial
$$
	f_i(X, Y_1, \dots, Y_{m-1})g_j (X, Y_1, \dots, Y_{m-1}) -  g_i(X, Y_1, \dots, Y_{m-1})f_j(X, Y_1, \dots, Y_{m-1})
$$
is nonzero for some choice of $i$ and $j$.  Indeed, if this were not the case, then we would find that
\begin{align*}
	f_i(X, Y_1, \dots, Y_{m-1}) g(X, Y_1, \dots, Y_{m})
		&= \sum_{j = 0}^d f_i(X, Y_1, \dots, Y_{m-1}) g_j(X, Y_1, \dots, Y_{m-1})Y_m^j\\
		&= \sum_{j = 0}^d g_i(X, Y_1, \dots, Y_{m-1}) f_j(X, Y_1, \dots, Y_{m-1})Y_m^j\\
		&=g_i(X, Y_1, \dots, Y_{m-1}) f(X, Y_1, \dots, Y_{m})
\end{align*}
for all $i$.  As $g$ and $f$ are irreducible and the only polynomials
on the left and right sides of the equation containing the variable
$Y_m$, we conclude that they are unit multiples of each other, which
contradicts the hypothesis of relative primality.

Therefore, let $T_1$ be the set of all $K$-rational points on the affine variety 
$$
	V_K(\{f_ig_j - g_if_j : 0\leq i < j \leq \dg_{Y_m}(f)\}).
$$  Since one of the polynomials in the defining set is nonzero, the affine variety is a proper closed variety, which implies that $T_1$ is a thin set by definition.  By construction, the set $T = T_0\cup T_1$ is the desired thin set.
\end{proof}

%
\section{Rank of a Formula}
\label{sec:rank}

The goal of this section is to define a notion of rank for existential formulas in the language of fields,
using degrees of polynomials and dimensions of varieties, as well as
the number of $\exists$-quantifiers used.  Certain formulas will have the same
rank, just as certain polynomials have the same degree. Crucially, the ranks are well-ordered.

\subsection{A useful well-ordering}

\begin{defn}
\label{defn:*order}
Let $(\L,<)$ be a linear order.  For a finite tuple $(a_0,\ldots,a_n)\in\L^{<\omega}$,
write $\avec^*$ for the tuple of the same $(n+1)$ elements (including repetitions) arranged
in $<$-descending order: $\avec^*=(a_{\alpha(0)},\ldots,a_{\alpha(n)})$ where $\alpha$
is a permutation and $a_{\alpha(i+1)}\leq a_{\alpha(i)}$ for all $i<n$.
Write $\avec =^* \bvec$ just if $\avec^*=\bvec^*$.

Then the $*$-order $(\L^*,<^*)$ is the lexicographic order $<^*$ (defined using $<$ on individual coordinates)
on the set $\L^*$ of $=^*$-equivalence classes in $\L^{<\omega}$.
To be clear:  if $\avec^*$ is a proper initial segment of $\bvec^*$,
then $\avec^* <^*\bvec^*$.
\end{defn}
Equivalently, one can view the elements of $\L^*$ as finite multisets of elements of $\L$,
with the elements of each multiset listed in $<$-nonincreasing order.
\begin{lemma}
\label{lemma:*wellorder}
If $(\L,<)$ is a well order, then so is $(\L^*,<^*)$.
\end{lemma}
\begin{proof}
Clearly $<^*$ is a linear order.  If it were not a well order, there would be a least $a\in\L$
such that some infinite $<^*$-descending sequence begins with an $\avec^*$ whose
greatest element is $a$.  Choose such an $\avec^*=(a^k,a_1,\ldots,a_n)$,
in nonincreasing order with $a_1<a$ after $a$ appears $k$ times,
with $k$ as small as possible (and allowing $n=0$).  Then the infinite
descending sequence beginning with this $\avec^*$ can only have finitely many
terms that begin with $a^k$, for if there were infinitely many, then by ``chopping off''
the $a^k$ from each term, we would get an infinite sequence contradicting the choice of $a$.
But then, immediately after the last term beginning with $a^k$ comes a term
beginning with $a^j$ for $j<k$, and this term also begins an infinite descending sequence in $\L^*$,
contradicting either the minimality of $k$ (if $j>0$) or the minimality of $a$ (if $j=0$).
\end{proof}

\subsection{Definition of rank}

We present an explicit way to put a well-ordering on the set of existential formulas with parameters in any given field.  This is done by associating a \emph{rank} to every existential formula.  

Every existential formula $\alpha(X)$ can be written in disjunctive normal form
$$\alpha(X) = \exists\Yvec(\alpha_1 \vee \alpha_2\vee\cdots\vee\alpha_n),$$
where each $\alpha_i(X, \Yvec)$ is a conjunction of equations and inequations.
Bringing the existential quantifiers inside the disjunctions and discarding
any unused quantifiers, every existential 
formula can be rewritten as
$$ ((\exists Y_1\cdots\exists Y_{m_1})\alpha_1) \vee \cdots\vee((\exists Y_1\cdots\exists Y_{m_n})\alpha_n),$$
where all variables $Y_1,\dots, Y_{m_i}$ appear in $\alpha_i$.  
One can also easily rearrange
any $\alpha_i(X,\Yvec)$ into a conjunction of the form 
$$f_1(X,\Yvec)=\cdots=f_k(X,\Yvec)=0~\&~
g(X,\Yvec)\neq 0.$$
Only one inequation $g\neq 0$ is needed, as several $g_i(X,\Yvec)$
could be multiplied together. It is allowed for $g$ to be the constant $1$.
We call an existential formula \emph{rankable} if it is given in the above format.
It is trivial to rearrange any existential formula into rankable format,
so in this paper every existential formula which appears is assumed to be rankable.

Before defining rank, we present a way to order tuples of polynomials.  Notice that this notion depends on a specific order for the variables.
\begin{defn} 
 For the variables
$X,Y_1,\ldots,Y_m$,  the \emph{multidegree} of a monomial $X^cY_1^{d_1}\cdots Y_m^{d_m}$  is
$(c,d_1,\ldots,d_m)$, and these $(m+1)$-tuples are ordered by the reverse
lexicographic order.  The \emph{multidegree $\mdeg{f}$ of a polynomial $f$} is the maximum
of the multidegrees of each monomial appearing (with nonzero coefficient) in it.
 \end{defn} 
 Observe that the linear order defined above on multidegrees is a well-ordering.


\begin{defn} 
\label{defn:rankable}
 A \emph{basic rankable formula} is an existential formula of the form
$$ \exists Y_1\cdots\exists Y_m~[f_1(X,Y_1,\ldots,Y_m)=\cdots=f_k(X,Y_1,\ldots,Y_m)=0~\&~
g(X,\Yvec)\neq 0],$$
and the \emph{rank} of such a formula is the triple
$$ \rk{\beta}=(m, e,
(\mdeg{f_1},\ldots,\mdeg{f_k})^*),$$
where the second component is the dimension $e$ of $V_{\Qbar}(\fvec) \cap D(g)$, as defined in Section \ref{sec:dim},
and the third component uses the $=^*$-classes of tuples of multidegrees, as in Definition
\ref{defn:*order}.
 \end{defn}
 
In this definition, we see that $V_{\Qbar}(\fvec)\cap D(g)$ is a subset of an ambient space of dimension $m + 1$. Therefore, the first coordinate of the definition of rank can be equivalently viewed as a measure of the dimension of this ambient space.  Additionally,
by Corollary \ref{cor:localization}, the base field does not matter in the definition of
the dimension $e$, so we will usually drop the $\Qbar$ from this notation.

We define an order $\prec$ on ranks of basic rankable formulas in forwards lexicographic order, meaning that
$$ (m, e, 
(d_1,\ldots,d_k)^*) \prec (m', e', 
(d'_1,\ldots,d'_{k'})^*)$$
if and only if one of the following holds:
\begin{itemize}
\item
$m<m'$, i.e., the first formula uses fewer $\exists$-quantifiers; or
\item
$m=m'$ and $e<e'$, so the first formula defines an open variety of lesser dimension than the second; or
\item
$m=m'$ and $e=e'$ and
$(d_1,\ldots,d_k)^* <^* (d'_1,\ldots,d'_{k'})^*$, so the first formula uses polynomials of lower multidegree.
\end{itemize}

The least possible rank of a
(satisfiable) basic rankable formula is $(0,0,(1)^*)$, which is the rank of the quantifier-free formula
$X=x$ for any specific value $x$:  here $m=0$, $k=1$ and the variety, which has
a single component whose dimension is $0$, is defined by $f_1=X-x=0$
whose multidegree (in the single variable $X$, since $m=0$) is simply $1$.
(The variety defined by $0=0$ has dimension $1$, so the formula $0=0$ has higher rank.)

Let $\mathcal R$ denote the set of all possible ranks of basic rankable
formulas.  Then $(\mathcal R, \prec)$ is a well-ordering.
(The third component of $\prec$
is well-ordered
by Lemma \ref{lemma:*wellorder}.)
Let $(R^*,\prec^*)$ be the result of applying Definition \ref{defn:*order} to $(\mathcal R, \prec)$.

Observe that an existential formula is rankable if and only if it is the finite
disjunction of basic rankable formulas.

\begin{defn}
\label{def:rank}
If $\alpha = \vee_{i=1}^r \beta_i$ is a rankable formula, the \emph{rank} of 
$\alpha$ is defined to be
$$\rk\alpha = ( \rk{\beta_1},\ldots, \rk{\beta_n})^* \in \mathcal R^*$$
\end{defn}

The rankable formulas can then be compared using the
ordering $\prec^*$.
By Lemma \ref{lemma:*wellorder}, $(\mathcal R^*, \prec^*)$ is a well-order.

\section{Minimal formulas and hypersurfaces}
\label{sec:hyp}

The well-ordering of ranks means that every nonempty set of
existential formulas has an element of least rank. For example, if
there exists an existential formula that defines $\OO_L$ in $L$, then
there is an existential formula $\alpha$ that accomplishes this which
has least rank among all such formulas.  Such a formula can be
considered a minimal successful formula.  This motivates the following
general definition.

  \begin{defn}
  For a field $L \subseteq\Qbar$ and an existential formula $\alpha(X)$ with coefficients from 
  $L$, we say $\alpha$ is $L$-\emph{minimal} if $\alpha$ has least rank among 
  all existential formulas $\alpha'$ for which
  \[\forall x (\alpha(x) \iff \alpha'(x))\]
  holds in $L$.
  \end{defn}

In order for the above to make sense, $\alpha'$ ranges only over those 
existential formulas which have parameters from $L$.
We will show that every $L$-minimal formula must take the form of a disjunction of 
formulas with two very simple formats: quantifier-free formulas, and formulas with only one equation and one inequation.

We will start by considering a general rankable formula, then minimize it as much as possible. First, we want to minimize the number of quantifiers, which is the first component of rank.  Clearly, we can eliminate the quantifier for any variable that does not appear in any polynomial of the formula.  The following simple lemma allows us also to remove any variables that appear in the inequation, but none of the equations.
 
 \begin{lemma} 
 \label{lem:forget} 
 Let $1\leq e < m$ and let $\delta(X)$ be the basic rankable existential formula 
 $$
 	\exists Y_1\cdots\exists Y_{m}~[f_1(X,Y_1,\dots,Y_e)=\dots = f_k(X, Y_1,\dots,Y_e)=0\neq g(X,Y_1,\dots,Y_m)],
$$
where $f_i\in F[X, Y_1, \dots, Y_e]$ and $g \in F[X,Y_1,\dots,Y_m]$ for some field $F$.

Then there are polynomials $g_1,\dots, g_r\in F[X, Y_1, \dots, Y_e]$
such that $\delta(X)$ is equivalent over $F$ to the disjunction of
formulas
$$
	\vee_{i = 1}^r \exists Y_1\cdots\exists Y_{e}~[f_1(X,\Yvec)=\dots = f_k(X, \Yvec)=0\neq g_i(X,\Yvec)].$$
  \end{lemma}
   \begin{proof} 
     Write out
     $g = \sum_{i = 0}^{d_m} g_{i}(X, Y_1, \dots, Y_{m - 1}) Y_m^i$ as
     a polynomial in $Y_m$.  Notice that if
     $(x, y_1, \dots, y_{m-1})\in \Qbar^{m}$ is any tuple, then there
     is a $y_m\in \QQ$ such that $g(x, y_1, \dots, y_m) \neq0$ if and
     only if $g_i(x, y_1, \dots, y_{m-1}) \neq0$ for some
     $0\leq i\leq d_m$.  Therefore, we can remove the quantifier for
     $Y_m$ and instead use a disjunction where $g$ is replaced by
     $g_j$ for $0\leq j\leq d_m$ in each formula.  By induction, this
     completes the proof.
\end{proof}

   To continue minimizing the number of quantifiers, we can take a more geometric perspective.  A basic rankable formula  $\beta(X)$ with $m$ quantifiers
 $$
 	\exists Y_1\cdots\exists Y_{m}~[f_1(X,\Yvec)=\dots = f_k(X, \Yvec)=0\neq g(X,\Yvec)]
$$
 corresponds to the projection to the $X$-coordinate of the points on the variety $D(g)\cap V(f_1,\dots, f_k)$.  Minimizing the number of quantifiers $m$ is equivalent to minimizing the dimension $m + 1$ of the ambient space where the variety lives.  If $k$ is large, then we expect the dimension $e$  of the variety to be much smaller than $m + 1$, and we can consider this ``wasteful," as it uses more variables than necessary.  
The following proposition uses a basic result of algebraic geometry to show that, in a special case with integral affine varieties, we only need $m = e$ quantifiers and a single equation to describe all but a lower-dimensional closed subset.  To complete the section, we will the show that this is enough to deduce the result in general.

\begin{prop} 
\label{prop:birat_open}
Let $F\subseteq \Qbar$ be a field and $\pp = (f_1,\dots,f_k) \subseteq F[X, Y_1,\dots, Y_m]$ a prime ideal. Define $\beta(X)$ to be the formula
$$\beta(X) = \exists Y_1,\dots,Y_m[f_1(X,Y_1,\dots,Y_m)=\dots=f_k(X,Y_1,\dots,Y_m)=0]$$ 
and set $e = \dim{V_F(\pp)}$.
If $\beta(X)$ is satisfied by infinitely many values of $X$ in $\Qbar$ and $e \leq m-1$, then after possibly reordering indices, there are polynomials $h\in F[X, Y_1,\dots, Y_e]$ and $s\in F[X, Y_1, \dots, Y_{e-1}]$ with $h$ irreducible and $s\not\in \pp$ such that $\beta(X)$
is equivalent to $\gamma_1(X) \vee \gamma_2(X)$ over $F$, using the formulas
\begin{align*}
  \gamma_1(X)&:~~~~~\exists Y_1\cdots\exists Y_{e}~[h(X,\ldots,Y_e)=0\neq
                  s(X,\Yvec)], \\
 \gamma_2(X)&:~~~~~\exists Y_1\cdots\exists Y_m~[s(X, \ldots, Y_{m}) = f_1(X,\ldots,Y_{m})=\cdots=f_k(X,\ldots,Y_{m}) = 0 ].
\end{align*}
 \end{prop}
 
 \begin{proof}
   Write $L = \Frac(F[X, Y_1,\dots, Y_m]/\pp)$. By Proposition
   \ref{prop:open}, we know that $e$ is equal to the transcendence
   degree of $L$ over $F$.  Since the images of
   $\{X, Y_1,\dots, Y_m\}$ generate $L$ over $F$, there is a
   transcendence basis consisting of a subset of these elements, and
   we can force $\bar X$ to be in this basis because $\bar X$ is not
   algebraic over $F$ \cite[Theorem VIII.1.1]{lang}.  Indeed, if
   $\bar X$ were algebraic over $F$, then it would be the root of a
   single-variable polynomial over $F$, and therefore $\beta(X)$ would
   only be solvable over $\Qbar$ by finitely many $X$, which is not
   the case by hypothesis.

   Reorder the variables so that
   $\{\bar X, \bar Y_1, \dots, \bar Y_{e-1}\}$ is a transcendence
   basis of $L$ over $F$. Write $L_0 = F(X, Y_1,\dots, Y_{e-1})$.
   Although a particular ordering of the variables is used when
   defining the multidegree component of rank in Definition
   \ref{defn:rankable}, we will produce lower-rank formulas purely in
   terms of quantifiers and dimension, and therefore the multidegree
   will not matter here. As $L$ is a finite separable extension of
   $L_0$, the primitive element theorem states that $L = L_0(\theta)$
   for a single element $\theta$.  Write $h\in L_0[Y]$ for the minimal
   polynomial of $\theta$. By clearing denominators if necessary, we
   can assume without loss of generality that
   $h\in F[X, Y_1,\dots, Y_{e-1}, Y]$ is an irreducible multivariable
   polynomial. Therefore, writing $\pp = (f_1,\dots, f_k)$, we have an
   isomorphism of fields:
$$
	L_0[Y_{e}, \dots, Y_m]/(f_1,\dots, f_k)
		\cong L
		\cong L_0[Y]/(h)
		 \cong \Frac(F[X, Y_1,\dots, Y_{e-1}, Y]/(h)) .
$$
Geometrically, this says that the integral affine variety $V_F(\pp)$ is birational to the hypersurface $V_F(h)$.
In fact, we can see that the two varieties contain isomorphic open sets, as follows.

Using the isomorphism of fields we can write
$Y_j = \sum_{\ell =0}^{N_j} c_{j,\ell} Y^\ell$ for each
$j=e,\ldots,m$, and $Y = \sum_{\vec{a}} d_{\vec a}Y_{e
}^{a_0}\dots
Y_{m}^{a_{m-e
}}$, where $c_{j,\ell}$ and $d_{\vec a}$ are elements
of $L_0$, and in particular not contained in $\pp$ because $L_0$ is a
subfield of the function field of $V_F(\pp)$.  Let $s$ be the products
of all denominators appearing in these terms.  Then these equations
give an isomorphism of the open sets $V_F(\pp)\cap D(s)$ and
$V_F(h)\cap D(s)$; see \cite[Lemma 3.7]{Liu}.  Moreover, the
$X$-coordinate of rational points is unchanged by the isomorphism
because we included $X$ in the transcendence basis.  As
$V_F(\pp) = (V_F(\pp)\cap D(s)) \cup V_F(\pp + (s))$, this proves the
claim that the formula $\beta$ is equivalent over $F$ to the
disjunction stated above.
 \end{proof}
 
 Next we show that minimal formulas all have a very convenient structure. 
  
 \begin{prop} 
 \label{prop:form}
If $\alpha(X) = \vee_{i =1}^r \beta_i(X)$ is a disjunction of basic rankable formulas and is  $L$-minimal for some field $L \subseteq \Qbar$, then each $\beta_i(X)$ has one of the following forms:
 \begin{enumerate}[(i)]
	 \item The quantifier-free formula $X=z_0$ for a fixed $z_0\in L$.
	 \item The ``hypersurface formula" $\exists Y_1\dots \exists Y_e~[ f(X, Y_1, \dots, Y_e) = 0 \neq g(X, Y_1, \dots, Y_e)]$ for an irreducible $f\in L[X, Y_1, \dots, Y_e]$ and a polynomial $g \in L[X, Y_1, \dots, Y_e]$.
 \end{enumerate}
  \end{prop}
  \begin{proof} 
 Let $\beta(X)$ be a fixed $\beta_i(X)$ which does not have the desired form. 
 Write $\beta$ in the form 
 $$
  \exists Y_1\cdots\exists Y_m~[f_1(X,\ldots,Y_{m})=\cdots=f_k(X,\ldots,Y_{m})=0\neq g(X,\Yvec)]
  $$
and consider the ideal $I = (f_1, \dots, f_k)$.   Define $e = \dim{V(I)\cap D(g)}$.
Without loss of generality, we can assume that each $f_i$ is irreducible.  Otherwise, if $f_1 = h_1h_2$ is a nontrivial factorization, then we could write $\beta$ as the disjunction of two formulas with $f_1$ replaced by $h_1$ and $h_2$, respectively, which have smaller multidegree.

Since $f_1$ is irreducible, $V_L(I)$ is a closed subset of the
integral affine variety $V_L(f_1)$ which has dimension
$\dim{V_L(f_1)} = m$ by Proposition \ref{prop:closed}.  In fact, we
see that either $V(f_1) = V(I)$, in which case we are done, or we have
$$
	e = \dim{V(I)\cap D(g)} \leq \dim{V(I)} < \dim{V(f_1)} = m.
$$  By assumption, we are in the latter case, and we will produce a set of formulas with parameters in $L$ which explicitly contradicts the minimality of $\alpha$.

The ideal $I$ has a primary decomposition
$I = \gq_1 \cap \dots \cap \gq_r$ where each $\gq_i$ is a primary
ideal associated to a prime ideal $\pp_i$. Indeed, the rational points
on $V(I)\cap D(g)$ are the same as the rational points on
$\cup_{i = 1}^r V(\pp_i)\cap D(g)$. Notice that the open set
$V(\pp_i)\cap D(g)$ might be empty for some $i$, but whenever it is
nonempty, $V(\pp_i)\cap D(g)$ has the same dimension as $V(\pp_i)$ by
Proposition \ref{prop:open}.

To summarize, we have shown that the formula $\beta(X)$ is equivalent to the disjunction $\vee_{i = 1}^r\delta_{\pp_i}(X)$ where each $\delta_{\pp_i}(X)$ is defined as a formula
$$\delta_{\pp_i}(X) = \exists Y_1,\dots, Y_m[p^i_1(X,\Yvec)=\dots=p^i_{n(i)}(X,\Yvec)=0\neq g(X,\Yvec)],$$
where $\pp_i = (p_1^i,\dots,p^i_{n(i)})$.
For each $i$, we will replace $\delta_{\pp_i}(X)$ itself with an
equivalent disjunction of basic rankable formulas, each of which 
has rank strictly smaller than $\beta$.  By definition, this contradicts 
the minimality of $\alpha$, and the proof will be done.

To this end, we analyze the primes $S = \{\pp_1, \dots, \pp_r\}$ and divide them accordingly.
Let $S_{\textrm{finite}}$ be the set of primes  $\pp\in S$ such that only finitely many elements of $L$ satisfy $\delta_{\pp}(X)$ in $F$, and let $S_\infty$ be all other primes of $S$.
 Partition $S_\infty= S_{\textrm{big}}\cup S_{\textrm{small}}$ where 
\begin{align*}
S_{\textrm{big}} &= \{\pp\in S_\infty \mid \dim{V(\pp)} = e \},\\
S_{\textrm{small}} &= \{\pp\in S_\infty \mid \dim{V(\pp)} < e \}.
\end{align*}

For any prime $\pp\in S_{\textrm{finite}}$, let
$\{z_{1}, \dots, z_{n}\}$ be the finite set of elements of $L$ which
satisfy $\delta_{\pp}(X)$ in $L$.  We may therefore replace
$\delta_\pp(X)$ with the disjunction of quantifier-free formulas
$\vee_{i = 1}^n (X - z_i)$.  Each of these quantifier-free formulas
consisting of a single-variable polynomial of degree 1 has the
smallest rank possible for a nontrivial basic rankable formula and
$\beta(X)$ has strictly larger rank.

For any $\pp\in S_{\textrm{small}}$, the formula $\delta_\pp(X)$ is already of smaller rank than $\beta$.  Indeed, the ambient space is the same, and the dimension is strictly smaller by definition.

For any $\pp\in S_{\textrm{big}}$, 
letting $\pp = (p_1,\dots,p_n)$,
we apply Proposition \ref{prop:birat_open} to see that 
$\exists \Yvec[p_1(X,\Yvec)=\dots=p_n(X,\Yvec)=0]$
 is equivalent to the disjunction of two formulas
\begin{align*}
  &\exists Y_1\cdots\exists Y_{e}~[f(X,\ldots,Y_e)=0\neq s(X,\Yvec)],\\
 & \exists Y_1\cdots\exists Y_m~[s(X, \ldots, Y_{m}) = p_1(X,\ldots,Y_{m})=\cdots=p_n(X,\ldots,Y_{m}) = 0 ],
\end{align*}
where $f\in L[X, Y_1,\dots, Y_e]$ is irreducible and $s\in L[X, Y_1,\dots, Y_{e-1}]$ is not contained in $\pp$.  Thus, $\delta_\pp(X)$ is equivalent to
 the disjunction of the following two formulas
\begin{align}
  &\exists Y_1\cdots\exists Y_{m}~[f(X,\ldots,Y_e)=0\neq g(X,Y_1,\dots,Y_m)s(X,Y_1,\dots,Y_{e-1})],\label{eq:1}\\
 & \exists Y_1\cdots\exists Y_m~[s(X, \ldots, Y_{e-1}) = p_1(X,\ldots,Y_{m})=\cdots=p_n(X,\ldots,Y_{m}) = 0 \neq g(X,\Yvec)].\label{eq:2}
\end{align}
By Lemma \ref{lem:forget}, we can replace the formula \eqref{eq:1} with a
disjunction of basic rankable formulas, each of which uses only $e$
quantifiers.  Since $e<m$, all these formulas have strictly smaller rank than
$\beta$.

On the other hand,   formula \eqref{eq:2} has $m$ quantifiers just like $\beta$, but 
we claim the associated variety has smaller dimension.  Indeed, we see that 
$$    \dim{V(\pp+(s)) \cap D(g)} \leq
	 \dim{V(\pp + (s))} 
		< \dim{V(\pp)}
		= \dim{V(\pp) \cap D(g)} = e,
$$
where the strict inequality follows
by Proposition \ref{prop:closed}.  Therefore this formula also has strictly
smaller rank than $\beta$.  This completes the proof.
\end{proof}

We can say more about the hypersurface formula appearing in the
previous result. First, we present a simple result on elements of the function field of an irreducible hypersurface.
  
  \begin{lemma} 
  \label{lem:lift_q}
  Let $F\subseteq \Qbar$ be a field and $f\in F[X, Y_1, \dots, Y_m]$ an irreducible polynomial whose degree in $Y_m$ is positive. If $\bar p/\bar q\in \Frac(F[X, Y_1, \dots, Y_m]/(f))$, then there are lifts of $p$ and $q$ to $F[X, Y_1, \dots, Y_m]$ such that $\dg_{Y_m}(q) < \dg_{Y_m}(f)$.
   \end{lemma}
    \begin{proof} 
Write $f = \sum_{i = 0}^d b_i Y_m^i$ where $b_i \in F[X, Y_1,\dots, Y_{m-1}]$. 
Choose arbitrary lifts $p_0, q_0\in F[Y_0,\dots, Y_m]$ of $\bar p$ and $\bar q$.  If $\dg_{Y_m} q_0 < \dg_{Y_m} f$, then we are already done.  Otherwise, define $p_1 = b_dp_0$ and $q_1 = b_d q_0$, which define the same fraction in the function field because $b_d, q_0\not\in (f)$.   Then the leading coefficient of $q_1$ is divisible by $b_d$, so we write it as $h_1b_d$ for $h_1\in F[Y_0,\dots, Y_{m-1}]$. Define $q_2 = q_1 - h_1 Y_m^{(\dg_{Y_m} q_1) - d}f_1$, and notice that $\dg_{Y_m}q_2 < \dg_{Y_m}q_1$. Continuing in this way, the claim follows.
\end{proof}

\begin{prop}\label{prop:absolutely-irreducible}
Suppose $L\subseteq \Qbar$ is a field
and $\beta(X)$ is a formula with parameters from $L$ of the following form
$$\beta(X) = \exists Y_1 \dots \exists Y_e[f(X,Y_1,\dots,Y_e)=0 \neq g(X,Y_1,\dots,Y_e)]$$
Suppose $\beta(X)$ is $L$-minimal.
Then $f$ is absolutely irreducible.
\end{prop}
\begin{proof}
First, it is clear that $f$ is irreducible in $L$; if it were reducible then $\beta$ could be 
equivalently expressed as the disjunction of two hypersurface formulas of 
strictly smaller rank.

Suppose for contradiction that $f$ is not absolutely irreducible.  We will use this 
fact to define $\{x \in L : \beta(x) \text{ holds in } L\}$ by a smaller rank formula using coefficients
from $L$.

Let $F \subseteq L$ 
be a number field containing all the coefficients which appear anywhere in $\beta$.
Let $K$ be a finite Galois extension of $F$ containing the coefficients of the absolutely irreducible factors of $f$ over $\Qbar$, and let $F' = K \cap L$.
Then $F \subseteq F' \subseteq K$, and $K$ is Galois over $F'$ because 
it was Galois over $F$. We remark that $F'$ is a subfield of $L$, and therefore $f$ is irreducible over $F'$.

For each of the finitely many number fields $E$ with $F' \subset E \subseteq K$, 
let $p_E \in  F'[Z]$ 
be a minimal polynomial for a primitive generator of $E$ over $F'$.  
Since $K$ is Galois over $F'$, none of these finitely many $p_E$ have a root in $L$.
Let $ h = \prod_{E : F' \subset E \subseteq K} p_E$.

We claim that $L$
has a lower-ranked formula $\phi$ with coefficients from $F'$ and
 with the property that 
for all $x \in L$,  
$\phi(x)$ holds over $L$ if and only if $\beta(x)$ does.

Let $M$ be the function field of $f$ over $F'$.  By Proposition \ref{prop:missinglemma},
$M$ therefore contains some element $z_0\in K \setminus F'$.
Moreover, $F'(z_0)$ is a subfield of $K$ which strictly contains $F'$.
So $F'(z_0)$ contains a root $z$  of $h$.

As an element of $M$, the root $z$ will be of the
form $\frac{p(X,\Yvec)+(f)}{q(X,\Yvec)+(f)}$, with $p,q\in F'[X,\Yvec]$.
We may view $p$ and $q$ as polynomials
 $p,q\in F'[X,Y_1,\ldots,Y_e]$, modulo the ideal $(f)$.
These polynomials will satisfy
$$ h\left(\frac{p(x,\yvec)}{q(x,\yvec)}\right)=0$$
whenever $(x,\yvec)$ is a solution to $f=0$ and $q(x,\yvec)\neq 0$.
 Therefore, 
 every solution
$(x,\yvec)\in L^{m+1}$ to $f=0$ has $q(x,\yvec)=0$.

By Lemma \ref{lem:lift_q}, we may choose our specific $q\in F'[X,\Yvec]$
so that $\dg_{Y_{e}}(q)<\dg_{Y_{e}}(f)$.  
Notice that $q\notin (f)$ because
$q+(f)$ is the denominator of an element of the function field, hence nonzero.
Below we will consider $q$ as a polynomial of degree $d$ in $Y_e$, writing
$q=\sum_{i\leq d} c_iY_e^i$ with all $c_i\in F[X,Y_1,\ldots,Y_{e-1}]$.
Without loss of generality, the leading nonzero coefficient $c_d$ does not lie in $(f)$.
If it happens that $Y_e$ does not appear in $q$, then $d=0$ and $c_0=q$.

But now we can use these facts to give a lower-ranked disjunction 
$\phi(X) = \gamma_0(X)\vee\gamma_1(X)$
which is equivalent to $\beta(X)$ in $L$.
Since $Y_e$ has lower
degree in $q$ than in $f$, the trick is to use the Euclidean algorithm here, 
using the leading term in the expansion $f=\sum_{i=0}^{d_1} Y_e^i \cdot b_i(X,Y_1,\ldots,Y_{e-1})$
and writing
$$  r(X,\Yvec) = c_d(X,\ldots,Y_{e-1})\cdot f(X,\Yvec) - 
b_{d_1}(X,Y_1,\ldots,Y_{e-1})\cdot Y_e^{d_1-d} \cdot q(X,\Yvec)$$
as a remainder with $\dg_{Y_e}(r)<\dg_{Y_e}(q)$.  Recall that the polynomial $c_d$
is the coefficient of $Y_e^d$ in $q$, hence does not involve $Y_e$.
Observe also that all coefficients of $r$ are in $F'$.

We claim that
in this situation, a tuple $(x,\yvec)\in L^{m+1}$ is a point on 
$V(f) \cap D(g)$ if and only if one of the following  conditions holds:
\begin{equation}\label{eqn:gamma0}
q(x,\yvec)=r(x,\yvec)=0\neq g(x,\yvec)\cdot c_d(x,y_1,\ldots,y_{e-1})
\end{equation}
or
\begin{equation}\label{eqn:gamma1}
f(x,\yvec)=c_d(x,y_1,\ldots,y_{e-1})=0\neq g(x,\yvec).
\end{equation}

To see the claim,
first let $(x,\yvec)$ be a point on $V(f) \cap D(g)$.  As shown above,
we must have $q(x,\yvec)=0$.  But the Euclidean equation shows that
$r(x,\yvec)=0$ as well, so the tuple satisfies one of the conditions, according to whether
$c_d(x,y_1,\ldots,y_{m-1})=0$ or not.
The converse of the claim follows by applying the Euclidean equation to the first condition, and the latter condition directly defines a subset of $V(f)\cap D(g)$.

The formulas $\gamma_0(X)$ and $\gamma_1(X)$ that we promised above are simply the
conditions in (\ref{eqn:gamma0}) and (\ref{eqn:gamma1}), 
each prefixed by $\exists Y_1\cdots\exists Y_e$.
Clearly these formulas have the same number of quantifiers as $\beta$.  The first formula corresponds to a subset $V(r, q)\cap D(gc_d)$ of $V(f)\cap D(g)$ because  $r + b_{d_1}Y_e^{d_1 - d}q = c_df$.  Hence the dimension of the subset cannot exceed the dimension of $V(f)\cap D(g)$.  However, $r$ and $q$ were constructed to have lower multidegree than $f$, so  $\gamma_0$ has strictly smaller rank than $\beta$.

On the other hand, the affine variety over $F'$ defined by the latter formula is a proper closed  subset of $V(f)$, hence 
$$
	\dim{V(f, c_d))\cap D(g)} 
		\leq	\dim{V(f, c_d))}  
		< \dim{V(f)}   
		= \dim{V(f))\cap D(g)} 
$$
showing that $\gamma_1$ has strictly smaller rank than $\beta$.
\end{proof}

Putting these results together yields the following normal form theorem 
for existential formulas in algebraic extensions of $\QQ$.

\begin{defn}
An \emph{absolutely irreducible hypersurface formula} is a formula of the form 
 $$\exists Y_1\dots \exists Y_e~[ f(X, Y_1, \dots, Y_e) = 0 \neq g(X, Y_1, \dots, Y_e)]$$ 
 for polynomials $f,g \in \Qbar [X, Y_1, \dots, Y_e]$,
	 where $f$ is absolutely irreducible and does not divide $g$.

\end{defn}
  
  \begin{thm}[Normal Form for Existential Definitions]\label{thm:normal-form}
 For any field $L \subseteq \Qbar$, if $A \subseteq L$ is existentially definable in $L$,
then $A$ is definable in $L$ by a formula of the form 
$$\alpha(X) = \vee_{i =1}^r \beta_i(X),$$ where each $\beta_i(X)$ has one of the following forms:
 \begin{enumerate}[(i)]
	 \item The quantifier-free formula $X=z_0$ for a fixed $z_0\in L$.
	 \item An absolutely irreducible hypersurface formula with coefficients from $L$ which is satisfied by infinitely many $x \in L$.
\end{enumerate}
  \end{thm}
 
  \begin{proof}
Apply Propositions \ref{prop:form} and \ref{prop:absolutely-irreducible}, 
plus the following two observations.  If $f$ divides $g$ in any of the hypersurface 
formulas, then that formula is unsatisfiable.  
If a hypersurface formula is satisfied by at most finitely many $x \in L$
(including if it is unsatisfiable), then
it could be replaced by a (possibly empty) 
disjunction of formulas of the form $X=z_0$, lowering 
the rank.
  \end{proof}

  \section{The meagerness of definability}
  \label{sec:main_thm}
 
 Recall that by identifying a subset of $\Qbar$ with its characteristic function,
 we can consider the set $\Subfields = \{L \subseteq \Qbar : L \text{ is a field}\}$
 as a subset of $2^{\Qbar}$, from which it inherits the product topology.  A basis for the 
 topology is given by the sets
 $$U_{\vec a, \vec b} = \{L \in \Subfields : a_1,\dots a_n \in L \text{ and } b_1,\dots,b_k \not\in L\}$$
 for any finite sequences of elements $\vec a, \vec b$ from $\Qbar$.
 If $\vec b$ is empty, we write simply $U_{\vec a}$.

Recall that \emph{Cantor space}, denoted $2^\omega$, is the set of infinite 
binary sequences with the product topology. 

\begin{prop}\label{prop:homeo}
The space $\Subfields$ is homeomorphic to Cantor space.
\end{prop}
\begin{proof}
Since $\Subfields$ is a closed subset of the Cantor-homeomorphic 
space $2^{\Qbar}$, it suffices to show that 
$\Subfields$ has no isolated points.  But it is clear that whenever 
$U_{\vec a, \vec b}$ is non-empty, there is $c \in \Qbar$ such that 
both $U_{(\vec a, c), \vec b}$ and $U_{\vec a, (\vec b, c)}$ are nonempty.
\end{proof}

   The upshot of Proposition \ref{prop:homeo} is a structure on the set $\Subfields$ which allows us to describe when a set is ``large" or ``small" in terms of topology.  In particular, we enlist the notions of meager sets and the property of Baire.
   \begin{defn} 
   A subset of a topological space is called \emph{nowhere dense} if its closure has empty interior, and \emph{meager} if it is the countable union of nowhere dense sets.  A topological space is \emph{Baire}\footnote{Some authors use the terminology \emph{Baire space} to refer to topological spaces with this property. However, we reserve the name Baire space for the particular topological space $\omega^\omega$, which is discussed in related papers, such as \cite{M19}, although we will not use it in this paper.} if every non-empty open subset is non-meager.
    \end{defn}
   
    Cantor space $2^\omega$ is Baire, and by Proposition \ref{prop:homeo} the same is true for $\Subfields$, which allows us to consider meager sets to be small.

\begin{defn}\label{defn:s-beta}
For any $Z \subseteq \Q$, and formula
$\beta(X)$ with coefficients $\vec a$ from $\Qbar$, we define $S_\beta(Z)$
to be the set of algebraic
fields in which $\beta$ defines a set disjoint from $Z$:
$$ S_\beta(Z)=\set{L\in U_{\vec a}}{ \set{ x \in L }{ \beta(x) \text{ holds over } L} \cap  Z= \emptyset}.$$
\end{defn}

\begin{defn}
For any $Z \subseteq \Q$, let $S(Z)$ denote the set
$$S(Z) = \bigcup_{\beta} S_\beta(Z)$$
where $\beta$ ranges over absolutely irreducible hypersurface formulas 
with coefficients from $\Qbar$.
\end{defn}

\begin{prop}
\label{prop:induction}
Let $Z$ be a subset of $\Q$ that is not thin in $\Q$.
Then $S(Z)$ is meager.  In particular, for every absolutely irreducible hypersurface formula 
$$\beta(X):~~\exists \Yvec [f(X,\Yvec) = 0 \neq g(X,\Yvec)]$$ 
with coefficients $\vec a$ from $\Qbar$,
the set $S_\beta(Z)$
is nowhere dense.
\end{prop}
\begin{proof}  Since $S_\beta(Z) \subseteq U_{\vec a}$, it suffices to show that 
$S_\beta(Z)$ is nowhere dense in $U_{\vec a}$.
Let $\vec b$ and $\vec c$ be any sequences of elements of $\Qbar$ such that 
$U_{(\vec a, \vec c),\vec b} \neq \emptyset$.  Let $F = \Q(\vec a, \vec c)$
and let $K = F(\vec b)$.
By the application of Hilbert's Irreducibility Theorem in Proposition \ref{prop:hit2}, there is a thin set $T\subseteq K^e$ such that for any $(x, y_1,\dots, y_{e-1})\in K^e\setminus T$, the polynomial $f(x, y_1, \dots, y_{e-1}, Y_e)$ is irreducible of degree $\dg_{Y_e}(f)$, and $g(x, y_1, \dots, y_{e-1}, Y_e)$  is not divisible by $f$.  
Because $K$ is a number field, $T_\QQ = T\cap \QQ^e$ is also a thin set in $\QQ^e$ by Proposition \ref{prop:fin-int}. 
Further, since $Z$ is not thin in $\QQ$, the thin set $T_\QQ$ does not contain all of $Z\times \QQ^{e-1}$ by Lemma \ref{lem:thin_prod}.
For any such tuple $(x, y_1,\dots, y_{e-1})\in Z\times \QQ^{e-1}$ outside this thin set, the irreducibility of $f(x, y_1,\dots, y_{e-1}, Y)$ over $K$ implies that adjoining to $F$ any root $y$ of $f(x, y_1,\dots, y_{e-1}, Y)$ will not generate any element of $K$:  we will have $F(y)\cap K=F$, by Lemma \ref{lemma:irreducibles}. Thus $U_{(\vec a, \vec c, y),\vec b}$ is nonempty.
Additionally, the divisibility condition implies that $g(x, y_1, \dots, y_{m-1},y)\neq 0$.
So for any $L \in U_{(\vec a, \vec c, y), b}$, $\beta(x)$ holds in $L$.
Therefore, $U_{(\vec a, \vec c, y),b} \cap S_\beta(Z) = \emptyset$.

There are only countably many $\beta$, so $S(Z)$ is a countable union of meager 
sets, and is thus meager.
\end{proof}

\begin{thm}
\label{thm:fullygeneral2}
%
If $Z\subset\Q$ is not thin, then the following is meager in $\Subfields$:
\begin{align*}
	S_Z
	&=\set{L\in\Subfields}{\text{some coinfinite $C\subseteq L$, universally definable in $L$, has~}C\cap\Q=Z}\\ 
	&=\bigcup_{A\subseteq \Qbar \colon A\cap \Q=Z} \{L \in \Subfields : A_L \text{ is coinfinite and universally definable in }L\}	
\end{align*}  
\end{thm}
\begin{proof}
Let $S = S(Z)$.  By Proposition \ref{prop:induction}, 
$S$ is meager.  Let $A \subseteq L$ be coinfinite with $A \cap \mathbb Q = Z$.
Suppose that $A$ is universally definable in $L$.  Then
$L \setminus A$ is existentially definable in $L$.  So by Theorem 
\ref{thm:normal-form}, $L\setminus A$ is definable in $L$ by a formula 
$\alpha = \vee_{i<r} \beta_i$ in normal form.  Because $L\setminus A$ is infinite 
and $r$ is finite, some $\beta_i$ must be an absolutely irreducible 
hypersurface formula, and 
$$\set{x \in L }{ \beta_i(x) \text{ holds over } L} \cap A = \emptyset.$$  
So $\set{x\in L}{ \beta_i(x) \text{ holds over }L} \cap Z = \emptyset$.
By definition, $L \in S_{\beta_i}(Z)$, as needed.
\end{proof}

As a corollaries we have the following.
\begin{thm}
\label{thm:meager_2}
The set of all fields $L\in \Subfields$ such that $\OO_L$ is either existentially or universally definable in $L$ is meager.
\end{thm}
\begin{proof}
Recall that $\ZZ = \QQ\cap \OO_L$ for all subfields $L$.
By Proposition \ref{thin}, neither $\Z$ nor $\Q\setminus\Z$ is thin in $\Q$.
Applying Theorem \ref{thm:fullygeneral2} to $Z=\mathbb Z$ and $Z=\mathbb Q \setminus \mathbb Z$, we see that there is a meager set $S$ such that if either $\OO_L$ or $L\setminus \OO_L$ is universally
definable in $L$, then $L \in S$.
\end{proof}

\begin{cor}
\label{cor:Z_2}
The set of fields $L\in \Subfields$ such that $\Z$ itself is either existentially or universally
definable in $L$ is meager.
\end{cor}

We can also use the same approach when considering the definability of
number fields.  Of course $\emptyset$ is a thin subset of $\Q$, so Theorem~\ref{thm:fullygeneral2} does not directly rule out an existential definition of $\Q$ in a generic algebraic extension $L$.  Nevertheless, we have the following.

\begin{cor}
\label{cor:Q}
 If $F$ is a number field, then the set of fields $L\in \Subfields$ containing $F$ such that $F$ has an existential definition in $L$ is a meager set.
\end{cor}
\begin{proof}
Let $S$ be the meager set guaranteed by Theorem \ref{thm:fullygeneral2} for 
$Z = \Z$.
By Park's  generalization \cite{P13} of a theorem of
 Koenigsmann  \cite{K16},  there is a quantifier-free formula $\phi(X,Y_1,\ldots,Y_n)$,
in the language of fields, such that $\exists \vec{Y} \phi(X,\vec{Y})$ defines the algebraic non-integers $F\setminus \OO_F$
in the field $F$.  In particular, it defines $\QQ\setminus \ZZ$ in $\QQ$ over $F$.  Now if $\gamma(Y)$ is existential and defines $F$ in $L$,
then the following formula with free variable $X$,
$$ \gamma(X)~\&~\exists \vec{Y} [~\gamma(Y_1)~\& \cdots \&~\gamma(Y_n)~\&~ \phi(X, \vec{Y})~], $$
is an existential definition of $F\setminus \OO_F$ in $L$.  So $(L\setminus F)\cup \OO_F$
is universally definable in $L$.  Since $((L\setminus F)\cup \OO_F)\cap \Q = \Z$,
we have $L \in S$.
\end{proof}

 \subsection{Computable fields whose algebraic integers are not one-quantifier definable}

Next we effectivize Theorem \ref{thm:meager_2} to obtain
 many computable algebraic extensions of $\QQ$ whose 
algebraic integers are not existentially or universally definable.

Our arguments below will require the decidability of absolute irreducibility.
Recall some standard terminology:  a computable field $E$
has a \emph{splitting algorithm} if the \emph{splitting set}
$S_E=\set{f\in E[T]}{f\text{~is reducible in~}E[T]}$ is decidable, and has a \emph{root algorithm}
if the \emph{root set} $R_E =\set{f\in E[T]}{f\text{~has a root in~}E}$
\label{rootset}
is decidable.  Notice that these are both stated for single-variable polynomials.
The next lemma is a specific case of the fact that splitting algorithms
can be extended to more variables.  
\begin{lemma}
\label{lemma:absirreducibility}
Fix any computable presentation of $\Qbar$.  Then it is decidable
which polynomials in $\Qbar[X_1,X_2,\ldots]$ are absolutely irreducible.
\end{lemma}
\begin{proof}
$\Qbar$ has a splitting algorithm, of course:  all polynomials in $\Qbar[T]$
of degree $>1$ are reducible.  The lemma now follows from another
theorem of Kronecker (found in \cite[$\S\S$ 58-59]{E84}), stating that whenever
a computable field $F$ has a splitting algorithm and $t$ is transcendental over $F$
(within a larger computable field), the field $F(t)$ also has a splitting algorithm.
The irreducible polynomials of $\Qbar[X_1,X_2]$ are precisely the irreducible polynomials of $\Qbar[X_1]$ along with the polynomials which are irreducible in $\Qbar(X_1)[X_2]$ and have no common factor among the coefficients lying in $\Qbar[X_1]$; see \cite[Theorem IV.2.3]{lang}.
Therefore, reducibility is clearly decidable
using Kronecker's result.  Thus we can decide reducibility
in $\Qbar[X_1,X_2]$, and one continues by induction
on the number $n$ of variables, noting that the resulting decision procedures
are uniform in $n$.
\end{proof}

Therefore, there is a computable listing $\beta_1,\beta_2,\dots$ of 
all absolutely irreducible hypersurface formulas.  Furthermore, we have the following 
effective version of Proposition \ref{prop:induction}.

\begin{prop}\label{prop:computable-extension}
Let $Z$ be a computable subset of $\QQ$ that is 
not thin in $\QQ$.  Then there is an algorithm which, given any 
absolutely irreducible hypersurface formula $\beta$ with coefficients $\vec a$,
 and any $\vec c, \vec b$ such that 
 $U_{(\vec a, \vec c), \vec b} \neq \emptyset$,
returns $y$ such that $U_{(\vec a, \vec c, y),\vec b}$ is non-empty
and has empty intersection with  $S_\beta(Z)$.
\end{prop}
\begin{proof}
The proof of Proposition \ref{prop:induction} shows that there is a tuple
$(x,y_1,\dots,y_{e-1},y) \in Z \times \QQ^{e-1} \times \Qbar$
which witnesses that $\beta(x)$ holds in each field extending $\Q(y)$ while
keeping $U_{(\vec a, \vec c, y), \vec b}$ non-empty.
So an algorithm can search all such $x,y_1,\dots,y_{e-1},y$ until 
it finds one.  This works because $Z$ is computable, and it 
is computable to check whether a given tuple from $\Qbar$ satisfies 
the polynomials appearing in $\beta$, and computable to check
whether $U_{(\vec a, \vec c, y), \vec b}$ is empty.
\end{proof}

\begin{thm}
\label{thm:computable_fields}
Let $Z\subseteq \Q$ be a computable subset which is neither thin nor co-thin. For every pair of $\Qbar$-tuples $(\vec a, \vec b)$, if $U_{\vec a, \vec b}$ is nonempty, then there is a computable 
$L \in U_{\vec a, \vec b}$ which enjoys the following property: 
If $A\subseteq L$ is any subset such that $A\cap \Q = Z$, then $A$ is neither existentially nor 
universally definable in $L$.  Moreover, every computable presentation
of $L$ has a splitting algorithm.
\end{thm}
\begin{proof}
We recursively define sequences $\vec a = \vec a_0, \vec a_1, \dots$ and $\vec b = \vec b_0, \vec b_1, \dots$
in stages as follows.  Recall that $\beta_1,\beta_2,\dots$ is a computable listing of 
all absolutely irreducible hypersurface formulas.  Let $c_1,c_2,\dots$ be a computable 
listing of all elements of $\Qbar$.

At stages of the form $s = 3t+1$, given $U_{\vec a_{s-1}, \vec b_{s-1}}$ nonempty, use 
Proposition \ref{prop:computable-extension} to find a $y$ such that 
$U_{(\vec a_{s-1},y),\vec b_{s-1}}$ is non-empty and 
disjoint from $S_{\beta_t}(Z)$.  Let
$\vec a_s = (\vec a_{s-1},y)$ and $\vec b_s = \vec b_{s-1}$.

At stages of the form $s = 3t+2$, use an analogous process to avoid 
$S_{\beta_t}(\QQ\setminus Z)$.

At stages of the form $s = 3t + 3$, consider $U_{(\vec a_{s-1},c_t), \vec b_{s-1}}$
and if it is nonempty, set $\vec a_s = (\vec a_{s-1},c_t), \vec b_s = \vec b_{s-1}$.
Otherwise, set $\vec a_s = \vec a_{s-1}$ and $\vec b_s = (\vec b_{s-1},c_t)$.

Because $\Subfields$ is a compact topological space by Proposition \ref{prop:homeo} and the family of closed subsets $\{U_{\vec a_s, \vec b_s} : s\geq 0\}$ has the finite intersection property, it follows that the intersection $\cap_{s} U_{\vec a_s, \vec b_s}$ is nonempty. Moreover, the intersection is a singleton because after stage $3t+3$, the element $c_t$ is included in either all or none of the fields in $U_{\vec{a_s}, \vec{b_s}}$ by construction.
In particular, this unique field is $L = \{ a \in \Qbar : a \text{ appears in some } \vec a_s\}$. 

The field $L$ is computable because by stage $3t+3$ it has 
been decided whether $c_t$ is included.  If $A\subset L$ is any subset with $A\cap \Q = Z$, then  
Theorem \ref{thm:fullygeneral2} implies that $A$ is neither existentially nor universally definable in $L$ because the construction of $L$ explicitly avoids the sets $S(Z)$ and $S(\Q\setminus Z)$ by definition.

The splitting algorithm for $L$ follows from Rabin's Theorem (see \cite{R60}),
since $L$ is given as a decidable subfield of (our computable presentation of) $\Qbar$.
Finally, whenever $L\cong\widetilde{L}$ are computable algebraic fields,
their splitting sets are Turing-equivalent, so all computable presentations of $L$
have splitting algorithms.
\end{proof}

As an application, because $\Z$ is neither thin nor co-thin in $\Q$, the following is immediate.
\begin{cor}
 For every pair of $\Qbar$-tuples $(\vec a, \vec b)$, if $U_{\vec a, \vec b}$ is nonempty, then there is a computable 
$L \in U_{\vec a, \vec b}$ such that $\OO_L$ is neither existentially nor 
universally definable in $L$.  Moreover, every computable presentation
of $L$ has a splitting algorithm.
\end{cor}

\subsection{ The topological space of algebraic extensions of $\QQ$ up to isomorphism}  
\label{sec:top}

The questions of definability we have considered have the same answer over 
isomorphic fields.  
Although $\Subfields$ contains at least one isomorphic copy of every possible
algebraic extension of $\Q$,  it contains exactly one copy of an algebraic extension $L$ of $\QQ$ if and only if $L$ is Galois over $\QQ$. A number field $F$ of degree $n$ is isomorphic to at most $n$ fields in $\Subfields$, but there are some infinite non-Galois extensions of $\QQ$ which are isomorphic to uncountably many elements in $\Subfields$.
Therefore, given the isomorphism invariance of the property under consideration,
one might wonder if the results of the previous section have been 
skewed by the fact that some isomorphism classes are more represented 
in $\Subfields$ than others.

Thus it is also of interest to consider the collection of algebraic extensions of $\Q$ 
up to isomorphism as a topological 
space, as was done in \cite{M19}.  We denote this set by $\Subfields/\!\cong$.
 From the perspective of number theory, the set $\Subfields/\!\cong$ can be identified as a quotient of $\Subfields$ by the absolute Galois group $G = \Gal{\Qbar}{\Q}$, which equates isomorphic fields.  The topology on 
 $\Subfields/\!\cong$ is the 
 quotient topology which it inherits from $\Subfields$.  
  
 Alternatively, from the perspective of computability theory, 
 one could begin with the space $\mathcal{ALG}_0^*$
 of all possible presentations of 
 algebraic extensions of $\Q$ in a certain language. This is done 
 in \cite{M19} and the relevant language in this case is the language of rings enlarged to include additional predicates for the existence of roots of monic one-variable polynomials.  
 Equating isomorphic fields and taking the quotient topology
 leads to the space $\mathcal{ALG}_0^*/\!\cong$, which coincides with 
 $\Subfields/\!\cong$ despite various differences between $\mathcal{ALG}_0^*$
 and $\Subfields$.  For example, in $\mathcal{ALG}_0^*$, every isomorphism 
 class is represented with uncountably many copies.  
 For details about $\mathcal{ALG}_0^*$, we 
 refer the reader to \cite{M19}.

Returning now to $\Subfields/\!\cong$, 
observe that for any $U_{\vec a, \vec b}$, the following set is the smallest $G$-invariant
 subset of $\Subfields$ containing $U_{\vec a, \vec b}$.  It is also clopen, as there 
 are only finitely many images $\phi(\vec a), \phi(\vec b)$.
 $$GU_{\vec a, \vec b} := \{\phi(L) : L \in U_{\vec a, \vec b}, \phi \in G\} = \bigcup_{\phi \in G} U_{\phi(\vec a), \phi(\vec b)}$$
 It follows that the quotient map $q:\Subfields \rightarrow \Subfields/\!\cong$ is open 
 and the images of the sets $GU_{\vec a, \vec b}$ form a clopen basis
 for $\Subfields/\!\cong$. 
 
 \begin{prop} [Theorem 3.3, \cite{M19}]
 $\Subfields/\!\cong$ is homeomorphic to Cantor space.
 \end{prop}
 \begin{proof}
 The follows because $\Subfields/\!\cong$ is compact, has a countable clopen basis,
 and has no isolated points.  The last condition follows 
 because every non-empty $GU_{\vec a, \vec b}$
 contains at least two non-isomorphic fields.
  \end{proof}
   
Therefore, notions of meager and co-meager make sense in $\Subfields/\!\cong$.
We can easily transfer the all our results about $\Subfields$ to 
results about $\Subfields/\!\cong$ by replacing the sets $S_\beta(Z)$ 
with the following $G$-invariant sets.
\begin{defn}
For any $Z \subseteq \Q$, and any absolutely irreducible hypersurface
formula $\beta$ with coefficients from $\Qbar$, let 
$$GS_\beta(Z) = \bigcup_{\phi \in G} S_{\phi(\beta)}(Z),$$
where $\phi(\beta)$ denotes the result of applying $\phi$ 
to all coefficients appearing in $\beta$.  
\end{defn}
Since every $\phi \in G$ fixes every element of $\mathbb Q$,
each $GS_\beta(Z)$ is $G$-invariant. 

\begin{prop}\label{prop:GS_nowhere_dense}
 Let $\beta$ be an absolutely irreducible hypersurface 
formula with coefficients from $\Qbar$.
If $Z\subseteq \Q$ is not thin in $\mathbb Q$, 
then $q(GS_\beta(Z))$ is nowhere dense in $\Subfields/\!\cong$, where
 $q:\Subfields \rightarrow \Subfields/\!\cong$ 
is the quotient map.
\end{prop}
\begin{proof}
There are only finitely many coefficients in $\beta$, so only finitely
many possible outcomes for $\phi(\beta)$.  So $GS_{\beta}(Z)$ is a finite union of 
nowhere dense sets, and thus is nowhere dense. 
Additionally, since each $S_{\phi(\beta)}(Z)$ is closed, so is $GS_\beta(Z)$. 
Since $\Subfields \setminus GS_\beta(Z)$ is dense open
and $q$ is an open map, its image $q(\Subfields \setminus GS_\beta(Z))$
is dense open.  Therefore, by $G$-invariance of $GS_\beta(Z)$,
$q(GS_\beta(Z))$ is nowhere dense. 
\end{proof}

\begin{thm}
\label{thm:fullygeneral3}
If any set $Z\subset\Q$ is not thin, then the following is meager in $\Subfields/\!\cong$:
$$ S_Z=\set{[L]_{\cong}\in \Subfields/\!\cong\ }{\text{some coinfinite $C\subseteq L$, universally definable in $L$, has~}C\cap\Q=Z}.$$

\end{thm}
\begin{proof}
Let $S = \bigcup_{\beta} q(GS_\beta(Z))$, where $\beta$ 
ranges over all absolutely irreducible hypersurface formulas.  By Proposition \ref{prop:GS_nowhere_dense}, 
$S$ is meager.  Let $A \subseteq L$ be coinfinite with $A \cap \mathbb Q = Z$.
Suppose that $A$ is universally definable in $L$.  Then
by the same argument as in Theorem \ref{thm:fullygeneral2}, $L \in S_{\beta}(Z)$ for some $\beta$.  So $L \in GS_\beta(Z)$, as needed.
\end{proof}

Therefore,
we have the following analogues of the results of the previous section.

\begin{cor}
\label{cor:up_to_isom}
The following sets are meager in $\Subfields/\!\cong$:
\begin{enumerate}
\item The set of isomorphism types of fields $L$ in which $\OO_L$ is existentially or universally definable.
\item The set of isomorphism types of fields in which $\Z$ is existentially or universally definable.
\item The set of isomorphism types of fields $L$ in which some number field $F\subset L$ is existentially definable.
\end{enumerate}
\end{cor}
\begin{proof}
These sets are all contained in the set $S$ guaranteed by Theorem \ref{thm:fullygeneral3} when $Z = \mathbb Z$.
\end{proof}

  It may seem equally natural to consider the Lebesgue measure
on Cantor space and transfer it to $\Subfields/\!\cong$,
using some computable homeomorphism such as that obtained in 
\cite[Theorem 3.3]{M19}.  This is attempted to some extent in
\cite{M19}, but the resulting measure is not canonical:  it depends to a great extent
on arbitrary choices that are made during the construction of the homeomorphism.  Indeed, the notion of
\emph{Haar-compatible measure}, put forth in \cite{M19}, has had
to be abandoned, as the reality is more complicated than the analysis
in that article recognized.  We hope to investigate this situation, and measure-theoretic perspectives in general, more fully
in the near future.

\end{document}